\documentclass{amsart}

\usepackage{enumerate}

\usepackage{hyperref}
\hypersetup{
    colorlinks=true,
    linkcolor=blue,
    filecolor=blue,      
    urlcolor=blue,
    citecolor=blue,
}
\usepackage[shortalphabetic]{amsrefs}
\usepackage{amsmath}
\usepackage{amsthm}
\usepackage{amssymb}
\usepackage[matha,mathx]{mathabx} 
\newcommand{\lversor}{\,\reflectbox{\ensuremath{\oslash}}\,}

\usepackage{wasysym}  
\usepackage[dvipsnames]{xcolor}
\usepackage{tikz}
\usetikzlibrary{arrows,cd}

\usepackage{url}

\numberwithin{figure}{section}

\numberwithin{equation}{section}

\newtheorem{mainthm}{Theorem}

\newtheorem*{thm*}{Theorem}
\newtheorem{lem}[equation]{Lemma}

\newtheorem{lemma}[equation]{Lemma}
\newtheorem{fact}[equation]{Fact}

\theoremstyle{definition}

\theoremstyle{remark}
\newtheorem{remark}[equation]{Remark}
\newtheorem{ex}[equation]{Example}

\DeclareMathOperator{\End}{End}

\DeclareMathOperator{\Aut}{Aut}

\DeclareMathOperator{\Der}{Der}
\DeclareMathOperator{\gl}{\mathfrak{gl}}

\DeclareMathOperator{\GL}{GL}
\DeclareMathOperator{\SL}{SL}

\DeclareMathOperator{\Cen}{Cen}

\DeclareMathOperator{\Nuc}{Nuc}
\DeclareMathOperator{\Inn}{Inn}

\DeclareMathOperator{\im}{im}

\newcommand{\C}{\mathbb{C}}
\newcommand{\M}{\mathbb{M}}
\newcommand{\bmto}{\rightarrowtail}

\newcommand{\K}{\mathbb{K}}

\newcommand{\ra}{\rangle}

\renewcommand{\phi}{\varphi}

\newcommand{\comp}[1]{\bar{#1}}

\newcommand{\bra}[1]{\langle#1|}
\newcommand{\ket}[1]{|#1\rangle}

\renewcommand{\hom}{{\rm Hom}}

\usepackage{cjhebrew}
\DeclareFontFamily{U}{rcjhbltx}{}
\DeclareFontShape{U}{rcjhbltx}{m}{n}{<->rcjhbltx}{}
\DeclareSymbolFont{hebrewletters}{U}{rcjhbltx}{m}{n}

\let\aleph\relax\let\beth\relax
\let\gimel\relax\let\daleth\relax

\DeclareMathSymbol{\aleph}{\mathord}{hebrewletters}{39}
\DeclareMathSymbol{\beth}{\mathord}{hebrewletters}{98}
\DeclareMathSymbol{\gimel}{\mathord}{hebrewletters}{103}
\DeclareMathSymbol{\daleth}{\mathord}{hebrewletters}{100}

\DeclareMathSymbol{\lamed}{\mathord}{hebrewletters}{108}
\DeclareMathSymbol{\mem}{\mathord}{hebrewletters}{109}
\DeclareMathSymbol{\ayin}{\mathord}{hebrewletters}{96}
\DeclareMathSymbol{\tsadi}{\mathord}{hebrewletters}{118}
\DeclareMathSymbol{\qof}{\mathord}{hebrewletters}{114}
\DeclareMathSymbol{\shin}{\mathord}{hebrewletters}{152}
\DeclareMathSymbol{\waw}{\mathord}{hebrewletters}{119}
\DeclareMathSymbol{\vavv}{\mathord}{hebrewletters}{119}
\DeclareMathOperator{\vav}{{{\ensuremath{\vavv}\,}}}

\renewcommand{\leq}{\leqslant}
\renewcommand{\geq}{\geqslant}
\renewcommand{\sl}{\mathfrak{sl}}

\newcommand{\Op}[2]{\mathfrak{Z}(#1,#2)}

\begin{document}

\title[Inner automorphisms of tensors]{Exact Sequences of Inner Automorphisms of tensors}
\author{Peter A. Brooksbank}
\address{
	Department of Mathematics\\
	Bucknell University\\
	Lewisburg, PA 17837
}
\email{pbrooksb@bucknell.edu}

\author{Joshua Maglione}
\address{
	Fakult\"at f\"ur Mathematik\\
	Universit\"at Bielefeld\\
	D-33501 Bielefeld, Germany
}
\email{jmaglione@math.uni-bielefeld.de}

\author{James B. Wilson}
\address{
	Department of Mathematics\\
	Colorado State University\\
	Fort Collins, CO 80523
}
\email{James.Wilson@ColoState.Edu}

\date{\today}

\thanks{This work was supported in part by NSF grants DMS-1620454 and DMS-1620362,
by the Simons Foundation $\#$281435, and by the Hausdorff Research Institute for Mathematics.}

\begin{abstract}
We produce a long exact sequence whose terms are unit groups of associative algebras
that behave as inner automorphisms of a given tensor.  
Our sequence
generalizes
known sequences for associative and non-associative algebras. In a manner similar to 
those, our sequence facilitates inductive reasoning about, and calculation of the groups 
of symmetries of a tensor.
The new insights these methods afford can be applied 
to problems ranging from
understanding algebraic structures to distinguishing entangled states in particle physics.
\medskip

\begin{center}
{\em In memory of C.C. Sims.}
\end{center}
\end{abstract}

\maketitle

\section{Introduction}
\label{sec:intro}
The purpose of this work is to provide tools to expose the symmetries of a tensor.
By a \emph{tensor} we mean a vector, $t$, that
can be interpreted as a multilinear map $\bra{t}\colon U_{\vav}\times \cdots\times U_1\bmto U_0$.
For instance, a $(d_2 \times d_1\times d_0)$-grid, $t=[t_{ij}^{k}]$, of scalars may
be interpreted  as a multilinear map $\bra{t}:\K^{d_2}\times \K^{d_1}\bmto \K^{d_0}$
evaluated on  $\ket{u_2,u_1}\in \K^{d_2}\times \K^{d_1}$ as follows:
\begin{align*}
	\bra{t} u_2,u_1\ra & = \left( 
		\sum_{i=1}^{d_2}\sum_{j=1}^{d_1} u_{2i}t_{ij}^{1} u_{1j},
		\ldots,
		\sum_{i=1}^{d_2}\sum_{j=1}^{d_1} u_{2i}t_{ij}^{d_0} u_{1j}\right)\in \K^{d_0}.
\end{align*}
Throughout, $\bmto$ will denote a multilinear map, while $\to$ will be reserved
for linear maps.
Tensors describe diverse structures, 
including distributive products in algebra,
affine connections in differential geometry, 
quantum entanglement in particle physics, 
and measurements and meta-data in statistical models.

A natural objective in the study of a tensor 
$\bra{t}\colon U_{\vav}\times \cdots\times U_1\bmto U_0$
is to discover properties invariant under change of basis.
We are particularly interested in one such invariant, 
namely its group
\begin{align}\label{eq:RZ}
	\Aut(t) & = \left\{\alpha\in \prod_{a=0}^{\vav} \Aut_{\K}(U_a)~~ \middle |~~
   \forall u\in\prod_aU_a,\;   
   \alpha_0\langle t|u_{\vav},\ldots,u_1\rangle=
		\langle t|\alpha_{\vav}u_{\vav},\ldots,\alpha_1u_1\rangle\right\}
\end{align}
of symmetries. This group is difficult to compute so
it is helpful to break it up into an exact sequence of groups which are, in general, 
easier to construct.  


There are precedents in algebra for such an approach.
The main idea, for a given associative algebra $A$,
is to study $\Aut(A)$ by placing it in an exact sequence
\begin{align}
\label{eq:RZ}
	1 \to \Inn (A) \to \Aut(A)\to \mathcal{J} (A),
\end{align}
where $\Inn(A)$ is the group of inner automorphisms, and $\mathcal{J}(A)$ has
a natural representation that can be explored without knowing $\Aut(A)$.  For instance, in Skolem-Noether 
type theorems, $\mathcal{J}(A)$ is the group of Galois automorphisms of the center of $A$. General
{\em Rosenberg--Zelinsky sequences} relate $\mathcal{J}(A)$ to the module theory of $A$
via ideal class groups, Picard groups, and so forth~\citelist{\cite{AH}\cite{GM}\cite{BFRS}}.   

Variations for non-associative Lie and Jordan algebras were carried out by Jacobson and others \cite{Benkart-Osborn}.  
Here, the notion of ``inner'' automorphisms is not obvious. If $L$ is a Lie algebra, for example, the appropriate substitutes are 
from the units in the associative algebra generated by left actions $L_x\in \mathrm{End}(L)$, where $L_x(y)=[x,y]$.  
Factoring out the inner automorphisms of central simple algebras leaves a group naturally represented as a Galois group---and in more general cases, a group akin to the $\mathcal{J}(A)$ used in general Rosenberg-Zelinsky sequences.

For tensors, the concept of inner automorphisms is even less clear.  
Even for a bilinear map $\langle t| \colon U_2\times U_1\rightarrowtail U_0$, 
the left actions $L_x\in \hom(U_1,U_0)$ compose only when $U_0=U_1$. 
For tensors of higher valence, 
there are $\binom{\vav+1}{2}$ suitable analogues of ``left'' or ``right'' actions. 
A solution however is visible in an earlier extension 
of Skolem-Noether type theorems for bilinear maps by the third author~\cite{Wilson:Skolem-Noether}; 
cf.\ Section~\ref{sec:mat}. The result is that our sequences do not begin $0\to \Inn(t)\to \Aut(t)$ 
but rather extend to the left of $\Aut(t)$ with a long sequence of corrections. 
Once in place, one can essentially follow the treatment
in~\cite{Wilson:Skolem-Noether} for non-associative algebras
to extend the sequences to the right.

\subsection{Notation \& terminology}
Throughout, the Hebrew letter $\vav$ (\emph{vav} to evoke {\em valence}) is a nonnegative integer.
Set $[\vav]=\{0,\ldots,\vav\}$ and $\binom{[\vav]}{i}=\{A\subset [\vav]\mid
|A|=i\}$. 
For $A\subset[\vav]$, write $\comp{A}=[\vav]-A$
and $\comp{a}=[\vav]-\{a\}$. Let 
$\K$ be a commutative unital ring and let $U_{\vav},\ldots,U_0$
be finitely generated $\K$-modules. Define
\begin{align*}
	U_0\oslash U_1 & = \hom(U_1,U_0) & 
	U_0\oslash \cdots \oslash U_{\vav}
		& = (U_0\oslash \cdots \oslash U_{\vav-1})\oslash U_{\vav}.
\end{align*}
Then $(-)\oslash U_a$ is a functor on 
modules with (left) adjoint functor $(-)\otimes U_{a}$ giving rise to the following natural isomorphisms
of $\K$-modules:
\begin{align*}
	U_0\oslash (U_{\vav}\otimes\cdots\otimes U_1) 
		& \cong 	U_0\oslash \cdots \oslash U_{a-1}\oslash(U_{\vav}\otimes \cdots\otimes U_{a})\\
		& \cong	 U_0\oslash \cdots \oslash U_{a}\oslash(U_{\vav}\otimes \cdots\otimes U_{a+1})\\
		& \cong U_0\oslash \cdots \oslash U_{\vav}.
\end{align*}
A \emph{tensor space} is a $\K$-module $T$ equipped with a $\K$-module monomorphism
\begin{align}\label{eqn:mono}
	\bra{\cdot}:T\hookrightarrow U_0\oslash \cdots \oslash U_{\vav}.
\end{align}
An element $t\in T$ is a {\em tensor}, and
$\bra{t}\colon U_{\vav}\times\cdots\times U_1\bmto U_0$
is its associated multilinear map.
For $\ket{u}=\ket{u_{\vav},\ldots,u_1}\in\prod_{a\neq 0}U_{a}$, write 
$\bra{t} u\ra\in U_0$ to mean the evaluation of $\bra{t}$ at $\ket{u}$.
The set $\{U_0,\dots, U_{\vav}\}$ of modules is the \emph{frame} of $T$,
and $\vav$ is its {\em valence}. 
For brevity, we often write $S\subseteq U_0\oslash \cdots \oslash U_{\vav}$ 
to denote a set of tensors and its frame.

Put $\Omega:=\prod_{a\in [\vav]}\End(U_a)$, the ring of {\em transverse operators}
on the tensor space $U_0\oslash\cdots \oslash U_{\vav}$, where 
$\End(U_a)$ is the ring of $\K$-linear endomorphisms of $U_a$.
The group $\prod_{a\in[\vav]}\Aut(U_a)$ of 
invertible transverse operators is the group $\Omega^{\times}$ of units of $\Omega$.
For $\omega_a\in\End(U_a)$,
write $\ket{\omega_au_a,u_{\bar{a}}}$ to apply $\omega_a$ to
$u_a$ while leaving the other coordinates fixed. 
If, for each $a\in[\vav]$, the condition
$\langle t| u_a, U_{\comp{a}}\rangle=0$ implies $u_a=0$, then
$t$ is \emph{nondegenerate}; if
$U_0=\langle t|U_{\comp{0}}\rangle$ then $t$ is {\em full}.
We say $t$ is {\em fully nondegenerate} if it is full and nondegenerate,
and we lose no essential information by assuming all our tensors are 
of this type.

\subsection{Main results}
We adopt Albert's \emph{autotopisms} and Leger--Luks' 
\emph{generalized derivations}~\cite{LL:gen-der}  as the principal invariants to study.
For $S\subset U_0\oslash \cdots \oslash U_{\vav}$,
\begin{align}
\label{eq:define-der}
\Der(S) &=
\left\{ \omega\in \Omega
\;\middle |\;  \forall t\in S, \forall u\in\prod_{a\ne 0}U_a,~
\omega_0\langle t\mid u\rangle  = 
			\sum_{c\in [\vav]-0} \langle t\mid \omega_c u_c,u_{\comp{c}}\rangle
			\right\}
\end{align}
is the (Lie) algebra of {\em derivations} of $S$, and 
\begin{align}
\label{eq:define-aut}
\Aut(S) & =
\left\{ \alpha\in \Omega^{\times}
\;\middle |\;  \forall t\in S, \forall u\in\prod_{a\ne 0}U_a,~
\alpha_0\langle t\mid u\rangle  =   \langle t\mid \alpha_{\bar{0}} u_{\bar{0}}\rangle
			\right\}
\end{align}
is the group of {\em automorphisms} (also called {\em autotopisms}) of $S$. 
For $0<a<b\leq \vav$, put $\Omega_{0a}=\End(U_0)\times\End(U_a)$,
$\Omega_{ab}=\End(U_a)^{{\rm op}}\times\End(U_b)$, and define  
\begin{equation}
\begin{split}
   \Nuc_{ab}(S) & = \left\{ \omega\in \Omega_{ab}
   ~\middle|~ \forall t\in S,\;\forall u\in \prod_{c\ne 0}U_c,~
 		\langle t\mid \omega_a u_a, u_{\comp{a}}\rangle 
   = \langle t\mid \omega_b u_b, u_{\comp{b}} \rangle \right\} \\ 
   \Nuc_{0a}(S) & = \left\{ \omega\in \Omega_{0a} ~\middle|~ 
    \forall t\in S,\; \forall u\in\prod_{c\ne 0} U_c, ~
 		\omega_0\langle t\mid u\rangle 
   = \langle t\mid \omega_a u_a, u_{\comp{a}}\rangle \right\}
\end{split},
\end{equation}
the
{\em nuclei} of $S$.
The opposite ring in the first equation
ensures that both types of nuclei are associative rings.
For $A\subset [\vav]$, 
put $\Omega_A=\prod_{a\in A}\End(U_a)$. 
Define the {\em centroids} of $S$ 
to be the associative rings
\begin{equation}\label{eq:cents}
\begin{split}
   \Cen_A(S) & = \left\{ \omega\in \Omega_A ~\middle|~ 
   \forall t\in S,\; \forall u\in \prod_{c\ne 0} U_c,\;
      \forall a,b\in A, 
      \langle t\mid \omega_a u_a,u_{\comp{a}}\rangle =
   \langle t\mid \omega_b u_b, u_{\comp{b}}\rangle \right\}\\
   \Cen_{A\cup 0}(S) & = \left\{ \omega\in \Omega_{A\cup 0} 
      ~\middle|~ \forall t\in S,\; 
      \forall u\in\prod_{c\ne 0} U_c, \;\forall a\in A, 
		\omega_0\langle t\mid u\rangle =
      \langle t\mid \omega_a u_a, u_{\comp{a}}\rangle \right\}.
\end{split},
\end{equation}
where $A\subseteq [\vav]-0$.
The assumption that $S$ is fully nondegenerate
ensures that all centroids are commutative. 
For $2<k\leq \vav$, put
\begin{align}
\label{eq:amalgamate}
	\Nuc(S) & =\bigoplus_{A\in\binom{[\vav]}{2}} \Nuc_A(S) & 
	\Cen_k(S) & = \bigoplus_{A\in\binom{[\vav]}{k}} \Cen_A(S).
\end{align} 

Our main theorems, which generalize Rosenberg--Zelinsky sequences to 
derivations and autotopisms of tensors, are the following.

\begin{mainthm}
\label{thm:exact-sequencesA}
For each fully nondegenerate 
$S\subseteq U_0\oslash \cdots \oslash U_{\vav}$,
there is an exact sequence of $\K$-Lie algebras
\begin{align*}
0 \rightarrow \Cen_{\vav}(S) \rightarrow \cdots \rightarrow \Cen_3(S) \rightarrow \Nuc(S) \rightarrow \Der(S).
\end{align*}
\end{mainthm}


\begin{mainthm}
\label{thm:exact-sequencesB}
For each fully nondegenerate 
$S\subseteq U_0\oslash \cdots \oslash U_{\vav}$,
there is an exact sequence of groups
\begin{align*}
 1\to \Cen_{\vav}(S)^{\times}\to  \cdots \to \Cen_3(S)^{\times}\to
	 \Nuc(S)^{\times}
	   \to \Aut(S).
\end{align*}
\end{mainthm}

\subsection{Outline}
The paper is organized as follows.
In Section~\ref{sec:seq},  we describe more general operators than those used
in the sequences of Theorems~\ref{thm:exact-sequencesA} 
and~\ref{thm:exact-sequencesB}. In particular we demonstrate how our theorems 
constitute part of a general strategy to attach polynomial-based invariants to tensors, 
which are then analyzed by restricting to subsets of variables. Geometrically, this 
corresponds to 
taking sections and inspecting fibers.
Section~\ref{sec:exact} constructs the exact sequences in Theorems~\ref{thm:exact-sequencesA}
and~\ref{thm:exact-sequencesB}, and identifies 
a certain combinatorial property needed to prove exactness.  This property is examined
in isolation in Section~\ref{sec:Rihanna}, followed by the
proofs our main theorems in Section~\ref{sec:proofs}.
Section~\ref{sec:app-ex} provides several examples that help to illustrate the breadth and
power of our methods. These examples include detailed examinations of symmetries of 
tensors as well as a brief look at quantum particle entanglement, where two states are
distinguished using our exact sequences.

\section{Creating the sequences}
\label{sec:seq}
Theorems~\ref{thm:exact-sequencesA} and \ref{thm:exact-sequencesB}
concern groups and Lie algebras, yet feature exact sequences
whose terms are predominately associative algebras or their unit groups. 
Although this is a convenient 
conversion from a pragmatic viewpoint---associative algebras are better understood, 
structurally, than groups and Lie algebras---switching categories 
may seem to the reader unnatural.  
However, the conversion has a natural explanation
when viewed through a broader geometric lens.

To introduce our sequences and expose their purpose we
make use of a device introduced in  \cite{FMW:densors} that 
records operators on a tensor space using
polynomial identities.
This establishes a convenient algebro-geometric vocabulary.

\subsection{Operator sets}
In \cite{FMW:densors} the following definitions where introduced as a means to 
capture, in a uniform manner, a wide range of common tensor operators.  
We include enough detail here to prove our stated claims in this more general context.

Let $\{U_0,\dots, U_{\vav}\}$ be a frame of $\K$-modules,
$\Omega=\prod_{a\in[\vav]}\End(U_a)$ the ring of transverse
operators on this frame, and 
$\K[X] = \K[x_{\vav},\ldots,x_0]$. 
For $\omega\in \Omega$, $p(X)=\sum_{e} \lambda_{e} X^{e}\in \K[X]$, 
$t\in U_{\vav}\lversor \cdots \lversor U_0$,
and $\ket{u}=\ket{u_{\vav},\ldots,u_1}\in \prod_{a> 0}U_a$,
define
\begin{align*}
	\bra{t} p(\omega)\ket{u} &= 
		\sum_{e} \lambda_{e} \omega_0^{e_0}\bra{t} \omega_{\vav}^{e_{\vav}}u_{\vav},
			\ldots, \omega_1^{e_1}u_1\ra.
\end{align*}
For each set $S$ of tensors, and each $P\subset \K[X]$, let
$\langle S|P(\omega)\ket{U}$ be the subspace generated by
$\bra{t}p(\omega)\ket{u}$ as $t$ ranges over $S$, $p$ over $P$, and 
$\ket{u}$ over $\prod_{a>0} U_a$.  Then,
\begin{align}\label{eq:sat}
	\Op{S}{P} & = \left\{ \omega \in \prod_{a\in[\vav]} \End(U_a)\; \middle | \;
		\bra{S}P(\omega)\ket{U}=0\right\}
\end{align}
is the \emph{operator set} for the pair $S,P$.
It will help our intuition to consider the sets $\Op{S}{P}$ as geometries.  
Indeed, over an algebraically closed field these are 
algebraic zero-sets.  For other rings (such as finite fields, as required
of several problems in algebra) $\Op{S}{P}$ still has the 
structure of an affine $\K$-scheme \cite{FMW:densors}.  

\begin{remark}
\label{rem:the-polys}
The polynomials defining $\Op{S}{P}$ as an affine $\K$-scheme are 
derived from the formula for $P$, but they are in general
quite different from $P$.  For instance, the number of variables in $P$ is $\vav+1$,
whereas in general the polynomials describing $\Op{S}{P}$ involve $\sum_{a} (\dim U_a)^2$
variables.  Thus, when defining a function on $\Op{S}{P}$, one
should not expect its image or fibers to again be sets of the form $\Op{S}{P}$.
\end{remark}

\subsection{Fibers of restriction}
We describe a general approach to study the zero-sets $\Op{S}{P}$ for an arbitrary set of polynomials; 
eventually, we return to the cases with which 
we are most concerned here. 
We begin by restricting 
the operator sets $\Op{S}{P}$ to subsets of the frame.  
Fix $A\subset[\vav]$, and define the projection
\begin{align*}
	\Lambda_A & \colon \Op{S}{P}\to \prod_{a\in A} \End(U_a), &
		\Lambda_A(\omega_b\colon b\in[\vav]) &=  (\omega_a\colon a\in A).
\end{align*}
Write $\Op{S}{P}|_A$ for the image of $\Lambda_A$,
which  
as noted in Remark~\ref{rem:the-polys}
need not be an operator set. However, 
the fibers of $\Lambda_A$ are still comprised of operators that satisfy
the polynomials $P$ on tensors $S$.  These we might  describe
 as sets $\Op{S}{Q}$ for various $Q$ related to $P$.
Extending our notation slightly, for $\omega\in \Op{S}{P}$, write
\begin{align*}
	\Op{S}{P(\omega_A,X_{\comp{A}})} = \Lambda_A^{-1}(\omega) & = \{ (\tau_{\comp{A}},\omega_A)\mid
		\bra{S}P(\tau_{\comp{A}},\omega_A)\ket{U}=0\}.
\end{align*}
There creates two issues.  
First, the formula depends on $\omega_A$. We can largely 
ignore this issues, however, since fibers over generic points---those not lying on a proper subvariety---are
invariant for a fixed irreducible component.
Thus, we have a notion of generic fibers over the components of
$\Op{S}{P}$ that is independent of $\omega$.
Secondly, $P(\omega_A,X_{\comp{A}})$ is
partially evaluated at linear operators and thus is no longer a polynomial in $\K[X]$.
To get to a polynomial, it suffices to evaluate $P$ at 
$\omega_A=(\lambda_a 1_{U_a} \mid a\in A)$
with $\lambda_a\in\K$. If $\omega_A$ is generic, its fibers are the operator sets 
\begin{align*}
	P_A(X_{\comp{A}}) & = P(\lambda_A, X_{\comp{A}})\in \K[X_{\comp{A}}].
\end{align*}
In this way, a generic fiber is isomorphic to $\Op{S}{P_A}$ and can be regarded 
as an operator set, as in~\eqref{eq:sat}.   Indeed,
we will confine ourselves to cases where $\Op{S}{P}$ is a group and 
$\Lambda_A$ is a group homomorphism; here, $\lambda_a\in\{0,1\}$
will be a convenient choice for us to construct our sequences explicitly.

\subsection{Polynomials defining groups and algebras}
The following fact---which follows directly from 
definitions~\eqref{eq:define-der},
~\eqref{eq:define-aut} and~\eqref{eq:sat}---concerns 
two specific polynomials related 
to the derivation algebra and the automorphism group of $S$. 

\begin{fact}\label{fact:aut-der}
If $D(X)=x_{\vav}+\cdots+x_1-x_0$
and $G(X)=x_{\vav}\cdots x_1-x_0$, then
\begin{align*}
	 \Der(S)  & = \Op{S}{D} & 
	\Aut(S) & = \Op{S}{G}.
\end{align*}
\end{fact}

Hereafter, we assume $P$ is chosen so that $\Op{S}{P}$
is closed under one of two group operations: addition of endomorphisms,
or composition of automorphisms.
(A characterization of such $P$ is given
in \cite{FMW:densors}, but it follows easily from
Fact~\ref{fact:aut-der} that both $D(X)$ and $G(X)$ have this property.)
In the latter case we still write $\Op{S}{P}$ but consider only invertible endomorphisms.  
Let $\epsilon$ denote the appropriate identity ($0$ or $1$) for $\Op{S}{P}$, 
which we interpret naturally as a constant in $\K$.
Each projection map, $\Lambda_A$ is a group homomorphism.
Thus, the fibration we created from $\Lambda_A$ has
a generic fiber, since every fiber is a coset of the kernel.  In particular, for each $A\subseteq[\vav]$
we have an exact sequence
\begin{equation}\label{eq:zero-set-seq}
\{\epsilon\} \longrightarrow \Op{S}{P_A}|_A \longrightarrow \Op{S}{P} \overset{\Lambda_A}{\longrightarrow} \Op{S}{P}|_A \longrightarrow \{\epsilon\}.
\end{equation}
Translated into the language of Section~\ref{sec:intro}, we observe the origins
of our replacements for inner derivations and inner automorphisms.
\begin{fact}
For each $\emptyset \ne A\subseteq [\vav]$, there are exact sequences
\begin{center}
\begin{tikzcd}
	0 \arrow[r] & \Der_A(S)\arrow[r, hookrightarrow] 
		& \Der(S) \arrow[r,"\Delta_A"] & \Der(S)|_A\arrow[r] & 0\\
	0 \arrow[r] & \Aut_A(S)\arrow[r, hookrightarrow] 
		& \Aut(S) \arrow[r,"\Gamma_A"] & \Aut(S)|_A\arrow[r] & 0	
\end{tikzcd},
\end{center}
where $\Der_A(S)=\Op{S}{D_A}|_A$ and $\Aut_A(S)=\Op{S}{G_A}|_A$.
(For emphasis, when specializing to derivations and autotopisms, 
shall replace the restriction maps $\Lambda$ with $\Delta$ and $\Gamma$, respectively.)
\end{fact}

\subsection{Chains of derivations and automorphisms}
We obtain a global outlook by summing over all restrictions to sets of a common cardinality. 
In this way we have one parameter to consider instead of exponentially many
subsets. 

\begin{fact}
\label{fact:first-step}
For each $k\in[\vav]$, there exists group homomorphisms $\Lambda_i^k$ $(i=1,2)$
that make the following diagram commute, and 
ensure that $\ker(\Lambda_1^k) = \im(\Lambda_2^k)$.
\begin{center}
\begin{tikzcd}
	\displaystyle\coprod_{A\in{\binom{[\vav]}{k}}} \Op{S}{P_A}|_A\arrow[r,"\Lambda^k_2"]	
		& \Op{S}{P} \arrow[r, "\Lambda^k_1"]\arrow[swap, dr, "\Lambda_A"] &  
			\displaystyle\prod_{A\in\binom{[\vav]}{k}} \Op{S}{P}|_A\arrow[d, "\pi_A"]\\
	\Op{S}{P_A}|_A\arrow[ur, hookrightarrow]\arrow[u, hookrightarrow, "\iota_A"] 
		\arrow[r]
		& \{\epsilon\} \arrow[r] & \Op{S}{P}|_A
\end{tikzcd}
\end{center}
\end{fact}

Fact~\ref{fact:first-step} follows from the exact sequences in~\eqref{eq:zero-set-seq}.
Using $P=D$ we have the following corollary.  
\begin{fact}
There is an exact sequence
\begin{center}
\begin{tikzcd}
	\displaystyle\bigoplus_{A\in\binom{[\vav]}{k}}\Der_A(S)
		\arrow[r,"\Delta_2^k"] &
		\Der(S) \arrow[r, "\Delta_1^k"]
		& \displaystyle\prod_{A\in \binom{[\vav]}{k}} \Der(S)|_A
\end{tikzcd}
\end{center}
\end{fact}

In formulating a group analogue we face the problem that finite coproducts of
of groups are not isomorphic to products.  This will prevent us from extending
our sequence to the left in the case of $P=G$.  The solution to the problem is 
implied by Theorem~\ref{thm:exact-sequencesB}, where we swapped from 
groups to units in a ring.  

\begin{lemma}
\label{lem:unit-embedding}
For $0\leq a<b\leq \vav$  
there is a natural monomorphism
\begin{align*}
\Nuc_{ab}(S)^{\times} \to
	\Op{S}{G_{ab}}|_{ab}.
\end{align*}
\end{lemma}
\begin{proof}
If $a=0$ then $G_{ab}=x_b-x_0$ and the identity map provides the 
isomorphism.  Otherwise $0<a$, so $G_{ab}=x_b x_a-1$.  
It follows that if
$(\omega_a,\omega_b)\in \Op{S}{G_{ab}}|_{ab}$
then $(\omega_a^{-1},\omega_b)\in \Op{S}{x_b-x_a}$. 
\end{proof}

By composing the inclusion $\Nuc_{ab}(S)^{\times}\to \Op{S}{G_{ab}}$
with $\Op{S}{G_{ab}}\hookrightarrow \Op{S}{G}$ in the diagram of 
Fact~\ref{fact:first-step}, we can replace $\coprod$ with $\bigoplus$
in the category of rings. Then, restricting to groups of units, we obtain the following.
\begin{fact}\label{fact:nuke-aut}
There is an exact sequence
\begin{center}
\begin{tikzcd}
	\Nuc(S)^{\times}
		\arrow[r,"\Gamma_2"] &
		\Aut(S) \arrow[r, "\Gamma_1"]
		& \displaystyle\prod_{A\in \binom{[\vav]}{2}} \Aut(S)|_A.
\end{tikzcd}
\end{center}
\end{fact}

\begin{remark}\label{rem:nuc-to-der}
For $0<a<b\leq \vav$, $\Op{S}{x_a + x_b} = \Der_{ab}(S)$ 
is naturally isomorphic to 
$\Op{S}{x_a - x_b} = \Nuc_{ab}(S)$ as vector spaces. 
This implicit isomorphism explains how nuclei appear in Theorem~\ref{thm:exact-sequencesA} instead of derivations. 
\end{remark}

Hereafter, we treat all products as coproducts, and use superscripts $\omega^A$, for $A\subseteq[\vav]$, to record
the factor from which an operator is taken.
In this way we obtain explicit (additive) notation for the functions
$\Lambda_2^k$:
\begin{align}\label{eq:Lambda_2^k}
	\Lambda^k_2\left((\omega^A_{a}:a\in A) :A\in \binom{[\vav]}{k}\right)
		& = \sum_{A\in \binom{[\vav]}{k}} (\epsilon_{\comp{A}},\omega_A^A).
\end{align}
Now our goal is to extend the sequence to an exact sequence ending in $\{\epsilon\}$.

\subsection{Summary}
It is possible to state the exact sequences between the
algebras and $\Cen_k(S)$ and $\Nuc(S)$ by explicit formulas, and these
will be given below.  As noted, however, these maps are made
for the convenience of working with associative algebras, when in fact
that change in categories is unnatural.  For instance, Fact~\ref{fact:nuke-aut} depends
on a twisting of $\Nuc(S)$, and Remark~\ref{rem:nuc-to-der} is
a similar twisting to turn an associative ring into a Lie algebra.
Indeed, $\Nuc(S)$ and $\Cen(S)$ only become the rings we seek 
\emph{after} we restrict the operator sets $\Op{S}{P_A}$ to
$\prod_{a\in A} \End(U_a)$.
Later we reveal an exponential number of seemingly
arbitrary choices in signs that make the sequences.
Even so, as we hope we have demonstrated in this section,
there is a clear, canonical picture being interpreted through these choices.  
The following toy example provides a nice illustration
of the main ideas.

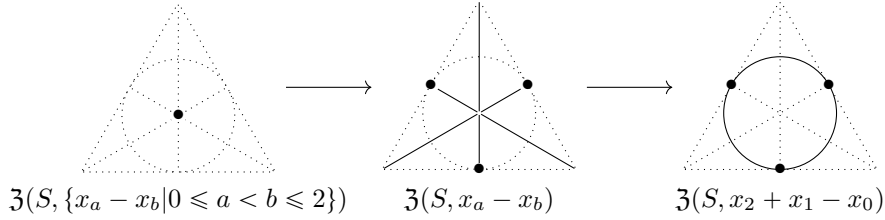
\begin{figure}[!htbp]
\begin{tikzpicture}
\node (Dt) at (0,-1.5) {$\mathfrak{Z}(S,x_2+x_1-x_0)$};
\node (D) at (0,0) {
	\begin{tikzpicture}[xscale=1.5,yscale=1.5,inner sep=0pt]
		\node (001) at (-0.86,-0.5) {};
		\node (010) at (0.86,-0.5) {};
		\node (011) at (0,-0.5) {$\bullet$};
		\node (100) at ( 0,1) {};
		\node (101) at (-0.43, 0.25) {$\bullet$};
		\node (110) at ( 0.43,0.25) {$\bullet$};
		\node (111) at (0,0) {};
		\draw[dotted] (001) -- (111) -- (110);
		\draw[dotted] (010) -- (111) -- (101);
		\draw[dotted] (100) -- (111) -- (011);
		\draw[dotted] (001) -- (011) -- (010);		
		\draw[dotted] (001) -- (101) -- (100);
		\draw[dotted] (010) -- (110) -- (100);				
		\draw (111) circle [radius=0.5];
	\end{tikzpicture}
};
\node (Nt) at (-4,-1.5) {$\mathfrak{Z}(S,x_a-x_b)$};
\node (N) at (-4,0) {
	\begin{tikzpicture}[xscale=1.5,yscale=1.5,inner sep=0pt]
		\node (001) at (-0.86,-0.5) {};
		\node (010) at (0.86,-0.5) {};
		\node (011) at (0,-0.5) {$\bullet$};
		\node (100) at ( 0,1) {};
		\node (101) at (-0.43, 0.25) {$\bullet$};
		\node (110) at ( 0.43,0.25) {$\bullet$};
		\node (111) at (0,0) {};
		\draw (001) -- (111) -- (110);
		\draw (010) -- (111) -- (101);
		\draw (100) -- (111) -- (011);
		\draw[dotted] (001) -- (011) -- (010);		
		\draw[dotted] (001) -- (101) -- (100);
		\draw[dotted] (010) -- (110) -- (100);				
		\draw[dotted] (111) circle [radius=0.5];
	\end{tikzpicture}
};
\node (Ct) at (-8,-1.5) {$\mathfrak{Z}(S, \{x_a-x_b | 0\leq a<b\leq 2\})$};
\node (C) at (-8,0) {
	\begin{tikzpicture}[xscale=1.5,yscale=1.5,inner sep=0pt]
		\node (001) at (-0.86,-0.5) {};
		\node (010) at (0.86,-0.5) {};
		\node (011) at (0,-0.5) {};
		\node (100) at ( 0,1) {};
		\node (101) at (-0.43, 0.25) {};
		\node (110) at ( 0.43,0.25) {};
		\node (111) at (0,0) {$\bullet$};
		\draw[dotted] (001) -- (111) -- (110);
		\draw[dotted] (010) -- (111) -- (101);
		\draw[dotted] (100) -- (111) -- (011);
		\draw[dotted] (001) -- (011) -- (010);		
		\draw[dotted] (001) -- (101) -- (100);
		\draw[dotted] (010) -- (110) -- (100);				
		\draw[dotted] (111) circle [radius=0.5];
	\end{tikzpicture}
};
\draw[arrows=->] (C) -- (N);
\draw[arrows=->] (N) -- (D);
\end{tikzpicture}
\caption{Geometric picture of the operator sets for the exact sequences 
when $\vav=2$.}\label{fig:plane}
\end{figure}
Defining $\bra{t}:\K\times \K\bmto \K$ with $\bra{t}u_2,u_1\ra = u_2u_1$, we have
\begin{equation}
\label{eq:data}
         \begin{array}{rcl}
	\Der(t)  & = & \Op{t}{x_2+x_1-x_0} = \{(\lambda_2,\lambda_1,\lambda_2+\lambda_1)\}~~\cong~~ \K^2\\
	\Op{t}{x_a-x_b} & = &  \{(\lambda_2,\lambda_1,\lambda_{0}) \mid \lambda_a=\lambda_b\}~~\cong~~ \K^2\\
	\Nuc_{ab}(t)  & = &  \Op{t}{x_a-x_b}|_{ab} = \{(\lambda,\lambda) \in \K\times \K\}~~\cong~~ \K\\
	\Cen(t) &  = & \{(\lambda,\lambda,\lambda)\}~~\cong~~ \K.
	\end{array}
\end{equation}
Figure~\ref{fig:plane} illustrates the sets in the projective space $\mathbb{P}^2$.
One can see from this figure and the data in~(\ref{eq:data}) that the nuclei as rings are not the same
as the operator sets, because of the restriction.  
Furthermore, the centroid is not contained, via inclusion, 
in the nuclei, but it is a subset of the full operator sets
$\Op{t}{x_a-x_b}$.  Similarly, the nuclei are not, in general, subsets of derivations. 
However, the full operator sets $\Op{t}{x_a-x_b}$ can be adjusted naturally to
$\Op{t}{x_a+x_b}$, and those intersect nontrivially with $\Der(t)$ from the 
embedding we described.  (In our picture 
because $|\K|=2$, where $-1=+1$, we do not see the distinction.)

For $\vav=2$ the sequence is described explicitly as follows.
Fix a projective line $\mathfrak{h}$ in $\mathbb{P}^{2}$.  Also fix
 points $\mathfrak{n}_{ab}$, $0\leq a<b\leq 2$
in general position on $\mathfrak{h}$, and a point $\mathfrak{c}$ not on $\mathfrak{h}$.
For example, if we use $\mathfrak{n}_{10}$, $\mathfrak{n}_{21}$, and $\mathfrak{c}$ to coordinatize
the (affine) plane then $\mathfrak{h}$ is given by a formula of the form
$\lambda_2 x_2+\lambda_1 x_1+\lambda_0 x_0=0$ with each $\lambda_a\neq 0$.  
This parametrizes an operator set $\Op{S}{\lambda_2 x_2+\lambda_1 x_1+\lambda_0 x_0}$ 
to which 
we attach a Lie algebra structure (though now the natural Lie brackets
are weighted by the coefficients $\lambda_a$).  We likewise use the lines
through each $\mathfrak{n}_{ab}$ and $\mathfrak{c}$ to define the sets of nucleus type
$\Op{S}{\mu_a^{ab} x_a-\mu_b^{ab} x_b}$, and the centroid lies
in their intersection.  After possibly deforming the scalars
of the nuclei, their operator sets intersect the derivation set at the points
$\mathfrak{n}_{ab}$. We obtain the desired embedding of $\Cen(S)$ into $\Nuc(S)$ 
with cokernel embedded into $\Der(S)$.  

All this occurs in generic terms and
can be reasoned for nonlinear structures such as the operator sets of groups.
Our specific interest in derivations and automorphisms are just two
natural demonstrations of an otherwise general technique.

\section{Exact sequences of groups and algebras}
\label{sec:exact}
We now specialize our discussion from Section~\ref{sec:seq} to the sequences in Theorems~\ref{thm:exact-sequencesA} 
and~\ref{thm:exact-sequencesB}. With our notation, 
we consider the case $k=2$, namely 
\[ 
\Lambda_2^2 : \bigoplus_{A\in{[\vav]\choose 2}}\Op{S}{P_A}|_A \rightarrow \Op{S}{P}. 
\]
We also restrict to $P\in\{D,G\}$;
in fact, we just consider $P=D$ and observe that the case
$P=G$ can be handled in a similar manner by working
just with invertible operators.
The function $\Lambda_2^2$ will now preserve further 
structure: it will be a Lie algebra homomorphism 
(or group homomorphism in the $P=G$ case).
Recall from Fact~\ref{fact:nuke-aut} and 
Remark~\ref{rem:nuc-to-der}, we must first 
adjust $\Lambda_2^2$ by twisting.

To accomplish this, we define an auxiliary function $\sigma$ that takes as input a pair of subsets of $[\vav]$ 
differing 
by one element, and returns a value in $\{-1,1\}$. 
Recall that subsets of $[\vav]$ of size 1 are written
without $\{\}$.
For $A=\{a,b\}$, put $C_A=x_a-x_b$, and 
define
$\Upsilon^2\colon \bigoplus_{A\in\binom{[\vav]}{2}} \Op{S}{C_A}|_A \rightarrow \Op{S}{P}$ by
\[ 
\bigoplus_{A\in\binom{[\vav]}{2}} (\omega_a^A : a\in A) \mapsto \left(\sum_{b\in[\vav]-a} \sigma(a, a\cup b)\cdot\omega_a^{a\cup b} : a\in [\vav]\right). 
\]
For $A\subseteq[\vav]$ of size at least two, let $C_A = \left\{ C_B : B\in\binom{A}{2}\right\}$. 
For $3\leq k \leq \vav+1$, if $A\in\binom{[\vav]}{k}$, from~\eqref{eq:cents} we have $\Op{S}{C_A}|_A=\Cen_A(S)$, so 
\[ 
\bigoplus_{A\in\binom{[\vav]}{k}} \Op{S}{C_A}|_A = \Cen_k(S). 
\]
For $2\leq k\leq \vav$, define $\Upsilon^{k+1}: \bigoplus_{A\in\binom{[\vav]}{k+1}} \Op{S}{C_A}|_A 
\rightarrow \bigoplus_{B\in\binom{[\vav]}{k}} \Op{S}{C_B}|_B$ by
\begin{align*}
    \displaystyle\bigoplus_{A\in\binom{[\vav]}{k+1}} (\omega_a^A : a \in A)\mapsto \bigoplus_{B\in\binom{[\vav]}{k}} \left( \sum_{a\notin B} \sigma(B, B\cup a) \cdot \omega_b^{B\cup a} : b\in B\right). 
\end{align*}

It remains to determine the conditions on $\sigma$ that ensure the functions $\Upsilon^k$ are well-defined and exact. 
The next result handles the former requirement.

\begin{lem}
\label{lem:well-defined}
For $3\leq k\leq \vav+1$, the maps $\Upsilon^k$ are well-defined.
The map $\Upsilon^2$ is well-defined if, and only if, 
\begin{equation}\label{eq:sigma-defined} 
\sigma(a, a\cup b) = \left\{ \begin{array}{ll} 
\sigma(b, a\cup b) & 0=a<b\leq \vav,\\
-\sigma(b, a\cup b) & 0<a<b\leq \vav.
\end{array}\right.
\end{equation}
\end{lem}

\begin{proof}
For $A=\{a,b\}$,
$
   \left(\epsilon_{\comp{A}}, \sigma(a, A)\cdot \omega_a^A, 
   \sigma(b, A) \cdot \omega_b^A\right) \in \Op{S}{P}
$
if, and only if, \eqref{eq:sigma-defined} holds (Fact~\ref{fact:first-step} and Lemma~\ref{lem:unit-embedding}). 
For $A\subseteq[\vav]$ of size $k\geq 3$, 
if $\omega = (\omega_a^A : a\in A)\in\Op{S}{C_A}|_A$, then 
\begin{align*} 
    \Upsilon^k(\omega) &= \bigoplus_{a\in A} \left( \sigma(A-a, A) \cdot \omega_b^A : b\in A-a\right) \\
    &= \bigoplus_{a\in A} \sigma(A-a, A)\cdot \left(\omega_b^A : b\in A-a\right).
\end{align*}
Thus, each summand is contained in $\Op{S}{C_{A-a}}|_{A-a}$. 
\end{proof}

For the next two lemmas, we assume that $\sigma$ satisfies~\eqref{eq:sigma-defined} so that the maps 
$\Upsilon^k$ are well-defined for $2\leq k \leq \vav+1$. 

\begin{lem}
\label{lem:sigma-hom}
Let $S$ be a fully nondegenerate tensor space. For all $3\leq k\leq \vav+1$, $\Upsilon^k$ is a homomorphism,
and $\Upsilon^2$ is a homomorphism if, and only if, 
\begin{align*}
(\forall 0\leq a<b\leq \vav) & & 
\sigma(a, a\cup b) = \left\{ \begin{array}{ll} 1 & a=0, \\ -1 & a>0. \end{array} \right. 
\end{align*}
\end{lem}

\begin{proof}
If $S$ is fully nondegenerate, then for $3\leq k\leq \vav+1$ and $A\in\binom{[\vav]}{k}$, 
$\Op{S}{P_A}|_A$ is a commutative ring, and so it has trivial Lie bracket and abelian unit group. 

We may therefore assume $k=2$. We give a proof for derivations (where $P=D$);
by Lemma~\ref{lem:unit-embedding} the proof for autotopisms (where $P=G$) is similar.
First, consider $A=\{0,b\}$. By Lemma~\ref{lem:well-defined}, $\sigma(0, A)=\sigma(b,A)$. 
Recalling that $\Nuc_{0b}(S)\subseteq\End(U_b)\times\End(U_0)$, the map 
$\Upsilon^2_A$ is a homomorphism if, and only if, $\sigma(0,A)=\sigma(0,A)^2$. 
Next, if $0<a<b$, then $\Nuc_{ab}(S)\subseteq\End(U_a)^{{\rm op}}\times \End(U_b)$. 
As the $a$-coordinate is contained in the opposite ring, 
$\Upsilon^2_A$ is a homomorphism if, and only if, $\sigma(a,A)=-\sigma(a,A)^2$.
\end{proof}

The notation for the calculations required to prove the exactness of these maps is simplified if, 
for $C\subset[\vav]$ of order at most $\vav-1$ and $a,b\in\comp{C}$, we set 
\begin{equation}\label{eq:tau}
\tau(C,a,b) = \sigma(C, C\cup a)\sigma(C\cup a, C\cup\{a,b\}) + \sigma(C, C\cup b)\sigma(C \cup b, C\cup\{a,b\}).
\end{equation}
The next result establishes the properties of $\sigma$ needed to ensure that 
the maps $\Upsilon^*$ form a chain complex. 

\begin{lem}
\label{lem:chain-complex}
Fix $2\leq k\leq \vav$. 
Then $\Upsilon^k\circ \Upsilon^{k+1}=\epsilon$ if, and only if, for all $C\in\binom{[\vav]}{k-1}$ 
and distinct $a,b\in\comp{C}$, $\tau(C,a,b)=0$.
\end{lem}

\begin{proof}
For $2\leq k\leq \vav$, 
\begin{align*}
	 \Upsilon^k\circ \Upsilon^{k+1} & \left(\bigoplus_{A\in \binom{[\vav]}{k+1}} 
	 	(\omega_a^A: a\in A)\right) \\
	 & =
		\bigoplus_{C\in \binom{[\vav]}{k-1}}	\left(
		\sum_{b\not\in C} 
		\sigma(C,C\cup b)
		\sum_{a\not\in C\cup b} 
			\sigma(C\cup b,C\cup \{a,b\})
			 \omega_c^{C\cup \{a,b\}}
			: c\in C\right) \\
	  & =
		\bigoplus_{C\in \binom{[\vav]}{k-1}}	\left(
		\sum_{\{a,b\}\in \binom{[\vav]-C}{2}}
		\tau(C,a,b)
		 \omega_c^{C\cup\{a,b\}}
			: c\in C\right). \qedhere
\end{align*}
\end{proof}

We next show that the conditions in 
Lemma~\ref{lem:chain-complex}---which are clearly 
necessary for exactness of $\Upsilon^*$---are also 
sufficient. 

\begin{lemma}
\label{lem:exact}
Fix $2\leq k\leq \vav$.
If $\Upsilon^k\circ \Upsilon^{k+1}=\epsilon$, then $\ker(\Upsilon^k) = \im(\Upsilon^{k+1})$.
\end{lemma}

\begin{proof}
We will express $\Upsilon^k$ as a matrix with entries in $\{-1,0,1\}$.
The nonzero entries are determined by the function $\sigma$. 
Our approach is then to derive a suitable transition matrix $M_k$ 
and study instead $\Upsilon^k M_{k+1} M_{k+1}^{-1}\Upsilon^{k+1}$. 
First, we arrange the components of $\bigoplus_{A\in\binom{[\vav]}{k+1}}\Op{S}{C_A}|_A$ so that 
the first summand to be over $[k]$, and thus collect together those operators that
act on the same term in the frame:
\begin{align*}
	\bigoplus_{A\in \binom{[\vav]}{k+1}}(\omega_a^A:a\in A) 
		\mapsto \bigoplus_{a\in [\vav]} \left(\omega_a^{B\cup a} 
			: B\in \binom{[\vav]-a}{k}\right).
\end{align*}
Reordering $\Upsilon^2$ into these coordinates, we rewrite $\Upsilon^{k+1}$ 
in terms of its restrictions $\Upsilon^{k+1}_a$ to each $\End(U_a)$, as follows:
\begin{align*}
	\Upsilon^{k+1}\left(	\bigoplus_{A\in \binom{[\vav]}{k+1}}(\omega_a^A:a\in A) \right)
		& = \bigoplus_{a\in [\vav]} 
			\Upsilon_a^{k+1}\left(\omega_a^{B\cup a} 
			: B\in \binom{[\vav]-a}{k}\right).
\end{align*}

Secondly, we use a matrix to describe each $\Upsilon_a^{k+1}$ as it
naturally leads to the concept of echelonizing, which will identify
the image of each map.  Rather than work with a block diagonal matrix
$\Upsilon^{k+1}=\bigoplus_a \Upsilon_a^{k+1}$, 
we instead fix $a\in [\vav]$ and focus on the matrix for just that fixed coordinate.  
Thus, for each $k$ we define a $\{-1,0,1\}$-valued matrix $M_{k+1}$ 
whose rows range over $\binom{[\vav]-a}{k}$
and whose columns range over $\binom{[\vav]-a}{k-1}$. 
Fixing a total order (e.g.\ lexicographic) on the subsets of $[\vav]-a$, define
\begin{align*}
	\left[M_{k+1}^{(a)} \right]_{AB} & = \left\{
		\begin{array}{cc}
			\sigma(B,A) & B\subset A, \\
			0 & B\not\subset A.
		\end{array}\right.
\end{align*}
For $b=\min([\vav]-a)$, we
observe the following partition of $M_{k+1}^{(a)}$:
\begin{align*}
M_{k+1}^{(a)} & = 
\begin{array}{cc}
	& \begin{array}{cc} b\in B  & b\not\in B \end{array} \\
\begin{array}{c} b\in A \\ b\notin A \end{array} & 
\left[\begin{array}{c|c}	
Y_{k+1} & Z_{k+1} \\ \hline
0 & X_{k+1} \end{array} \right]
\end{array},
& 
Z_{k+1} & = \bigoplus_{a\in A - B} \sigma(B,B\cup a).
\end{align*}
The block of $0$'s in the lower left follows from the fact that those rows and
columns are indexed by subsets $A$ and $B$ for which $B\not\subset A$.
Similarly in the upper right block $B\subset A$ at exactly the
row $A=B\cup a$ and column $B$, for our fixed $a$.  As 
$\sigma(B,B\cup a)\in \{-1,1\}$, it follows that $Z_{k+1}$ is diagonal, invertible,
and of order at most $2$.  Letting $Z_1=1$ for the base case, we
now claim that for each $k$,
\begin{align}
\label{eq:claim}
	Y_{k+1} & = -Z_{k+1} X_k Z_{k}.
\end{align}
Indeed, as $Z_{k+1}$ is an involution we can rewrite this identity 
equivalently as
\begin{align}
\label{eq:key}
	Z_{k+1}X_k+Y_{k+1}Z_k & = 0.
\end{align}
Focusing just on non-zero coordinates, condition (\ref{eq:key}) translates precisely
to the condition $\tau(C,c,d)=0$, as in in~\eqref{eq:tau}. 
Thus, by Lemma~\ref{lem:chain-complex}, equation~\eqref{eq:claim} holds. 

Having identified the structure of the matrices, 
for each $a\in[\vav]$ we can echelonize the coordinates of the transform
$\Upsilon^{k+1}_a$ to identify the symbolic rank of each matrix. This
will ensure that the image equals the kernel. To simplify notation for inverses,
we put $N_k=M_k^{(a)}$, and compute
\begin{align*}
N_2^{-1} \Upsilon_a^2 
	& = 
	\begin{bmatrix} Z_2 & 0 \\ -X_2Z_2 & I \end{bmatrix}
	\begin{bmatrix} Z_2 \\ X_2 \\ \end{bmatrix}
	& = \begin{bmatrix} 1\\ 0 \end{bmatrix}\\
	\vdots & & \vdots \\
N_{k+1}^{-1}\Upsilon^{k+1}_a N_k & =
\begin{bmatrix} Z_{k+1} & 0 \\  -X_{k+1}Z_{k+1} & I \end{bmatrix}
\begin{bmatrix}
	-Z_{k+1} X_k Z_k & Z_{k+1}\\
	0 & X_{k+1}
\end{bmatrix}
\begin{bmatrix} Z_k & 0 \\ X_k & I \end{bmatrix}
& =
\begin{bmatrix}
	0 & I\\
	0 & 0
\end{bmatrix}\\
\vdots & & \vdots \\
\Upsilon^{\vav+1}_a N_{\vav} & = 
\begin{bmatrix} Z_{\vav} X_{\vav-1} Z_{\vav-1} & Z_{\vav} \end{bmatrix}
	\begin{bmatrix}
	Z_{\vav} & 0 \\ Z_{\vav}X_{\vav-1}Z_{\vav-1} & Z_{\vav} 
	\end{bmatrix}
	&	= \begin{bmatrix} 0 & 1 \end{bmatrix}
\end{align*}
Evidently the image (column span) of each matrix on the right is the kernel 
(right null space) of the one below it. Since each row-column operation is
carried out by unimodular transform (a permutation or transvection with $\pm 1$-values) 
it is possible to perform these operations symbolically on our coordinates.
Thus, elements in the kernel of $\Upsilon_a^{k}$ can be used to
write elements in the image of $\Upsilon_a^{k+1}$.  
\end{proof}

\subsection{Conditions on \boldmath$\sigma$}
\label{sec:conditions}
To recap, three properties are required 
of $\sigma$ that ensure the homomorphisms $\Upsilon^*$ 
form an exact sequence. First, by Lemma~\ref{lem:well-defined}, 
we require 
\[ \sigma(a, a\cup b) = \left\{ \begin{array}{ll} 
\sigma(b, a\cup b) & 0=a<b\leq \vav,\\
-\sigma(b, a\cup b) & 0<a<b\leq \vav,
\end{array}\right. \]
so that the maps are well-defined. Secondly, from Lemma~\ref{lem:sigma-hom} we require
\[ \sigma(a, a\cup b) = \left\{ \begin{array}{ll} 
1 & 0=a<b\leq \vav,\\
-1 & 0<a<b\leq \vav,
\end{array}\right. \]
to ensure the homomorphism property. 
Finally, by Lemmas~\ref{lem:chain-complex} and~\ref{lem:exact} 
we require that for all $C\subset [\vav]$ with order at most $\vav-1$, 
and for all distinct $a,b\notin C$, 
\[ 
\sigma(C, C\cup a)\sigma(C\cup a, C\cup\{a,b\}) + \sigma(C, C\cup b)\sigma(C \cup b, C\cup\{a,b\}) = 0. 
\]

\section{Directed graphs}
\label{sec:Rihanna}
The goal of this section is to prove the existence of a 
function $\sigma$ satisfying the conditions laid out in 
Section~\ref{sec:conditions}.
This is accomplished by means of a directed graph  
$\mathcal{G}_{\vav}$ whose vertices are subsets of $[\vav]$.
Two vertices $A,B\subseteq [\vav]$ are adjacent if there exists $b\notin A$ such that $A\cup b=B$. 
Our objective is to define an {\em orientation} on $\mathcal{G}_{\vav}$, namely to assign
a direction to each edge in $\mathcal{G}_{\vav}$ that encodes the nonzero 
values of $\sigma$. A directed edge $C\cup a\to C$ can be thought of as ``down" in the underlying poset
and carries a value of 1, while $C\to C\cup a$ is ``up" and carries the value $-1$. 

To state the key result, we introduce some convenient terminology. 
We refer to an edge from $A\subseteq [\vav-1]$ to $A\cup\vav$
as a {\em controlling edge}.
For distinct $a,b\in[\vav]$ and $C\subseteq[\vav]- \{a,b\}$, denote by $\mathcal{D}(C, a, b)$ the subgraph of 
$\mathcal{G}_{\vav}$ induced on the four vertices labeled by $C$, $C\cup a$, $C\cup b$, and $C\cup\{a,b\}$. 
We refer to subgraphs $\mathcal{D}(C,a,b)$ as \emph{diamonds} of $\mathcal{G}_{\vav}$.
We say an orientation on $\mathcal{G}_{\vav}$ is \emph{oddly acyclic} if every diamond of 
$\mathcal{G}_{\vav}$ is acyclic with a path of length 3. 
For consistency, we say that every orientation on $\mathcal{G}_0$ is (vacuously) oddly acyclic.

Our first result collects the diamonds of $\mathcal{G}_{\vav}$ into three classes.
For $A\subseteq[\vav]$, we say a diamond $\mathcal{D}$ is contained in $2^{A}$, the power set of $A$, if the vertex labels of 
$\mathcal{D}$ are contained in $A$. 

\begin{lem}
\label{lem:diamonds}
For $\vav\geq 1$, every diamond of $\mathcal{G}_{\vav}$ is either contained in $2^{[\vav-1]}$, contained in 
$2^{[\vav]} - 2^{[\vav-1]}$, or contains exactly two controlling edges.
This partitions the set of diamonds of $\mathcal{G}_{\vav}$. 
\end{lem}

\begin{proof}
No diamond in $2^{[\vav-1]}$ or $2^{[\vav]}-2^{[\vav-1]}$ contains a controlling edge, so the three sets are disjoint. 
Suppose $\mathcal{D}$ is a diamond not contained in $2^{[\vav-1]}$ or $2^{[\vav]}-2^{[\vav-1]}$.
Because $2^{[\vav]}$ and $2^{[\vav]}-2^{[\vav-1]}$ are subset-closed, it follows that $\mathcal{D}$ has two vertices 
with labels in $2^{[\vav]}$ and two vertices with labels in $2^{[\vav]}-2^{[\vav-1]}$.
\end{proof}

We dedicate the key result in this section to the 
inspiration for its proof:

\begin{lem}[Rihanna's Lemma]
\label{lem:RiRi}
Suppose $\vav\geq 1$.
For every oddly acyclic orientation on the subgraph $\mathcal{G}_{\vav-1}$ of $\mathcal{G}_{\vav}$, 
and for every orientation on the controlling edges of $\mathcal{G}_{\vav}$, there exists a unique induced 
oddly acyclic orientation for $\mathcal{G}_{\vav}$.
\end{lem}

\begin{proof}
Think of
diamonds contained in $2^{[\vav-1]}$ as \emph{green diamonds}, and of those contained in 
$2^{[\vav]} - 2^{[\vav-1]}$ as \emph{yellow diamonds}.
Call a diamond \emph{controlling} if it contains two controlling edges. 
Observe that every non-controlling edge of $\mathcal{G}_{\vav}$ lies in a unique controlling diamond.

We construct the induced orientation 
from the orientations on $\mathcal{G}_{\vav-1}$ 
and the controlling edges of $\mathcal{G}_{\vav}$ as follows. 
For every $a\in[\vav-1]$ and $C\subseteq [\vav-1]-a$,
the (yellow) edge $y$ incident to $C\cup\{a,\vav\}$ and $C\cup\{\vav\}$ lies in a unique controlling diamond 
$\mathcal{D}(C,a,\vav)$. As a scholium to 
Lemma~\ref{lem:diamonds}, we observe that the orientation of exactly 
three edges of $\mathcal{D}(C,a,\vav)$ is determined by the oddly acyclic orientation on 
$\mathcal{G}_{\vav-1}$ and the orientation on the controlling edges of $\mathcal{G}_{\vav}$. 
Thus, there is a unique choice of orientation of $y$ such that $\mathcal{D}(C,a,\vav)$ is oddly acyclic.
Impose this orientation upon all yellow edges, 
and consider the resulting directed graph $\mathcal{G}_{\vav}$.

By Lemma~\ref{lem:diamonds} we must show that every yellow diamond is oddly acyclic.
Such a diamond has the form $\mathcal{D} = \mathcal{D}(C,a,b)$, where $\{\vav\} \subseteq 
C\subseteq[\vav]$ and $a,b\in [\vav-1]$ with $a\ne b$. 
Let $\mathcal{D}'=\mathcal{D}(C-\vav,a,b)$, 
a green diamond associated to $\mathcal{D}$.
There are exactly four controlling edges incident to both $\mathcal{D}$ and $\mathcal{D}'$. 
Two controlling edges incident to the same edges, $y$ in $\mathcal{D}$ and $g$ in 
$\mathcal{D}'$, point in the same direction if, and only if, the edges $y$ and $g$ point in opposite directions. 
This implies that $\mathcal{D}$ is oddly acyclic, 
since by hypothesis $\mathcal{D}'$ is oddly acyclic. 
The result now follows.
\end{proof}

\nameref{lem:RiRi} is illustrated in Figure~\ref{fig:RiRi} for $\vav=2$. 
The orientations for the green diamond and the controlling edges have been given, and 
by~\nameref{lem:RiRi} there is a unique choice of orientation for the yellow diamond so that $\mathcal{G}_2$ is oddly acyclic.

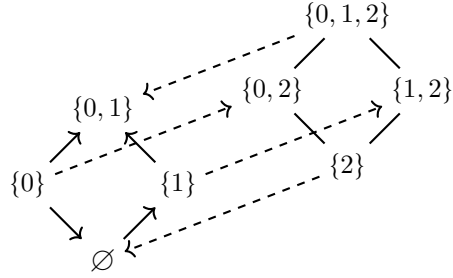
\begin{figure}[h]
\centering
\begin{tikzpicture}
    \node(E) at (0,0) {$\emptyset$};
    \node(0) at (-1,1) {$\{0\}$};
    \node(1) at (1,1) {$\{1\}$};
    \node(01) at (0,2) {$\{0,1\}$};
    %
    \node(2) at (3.25,1.25) {$\{2\}$};
    \node(02) at (2.25,2.25) {$\{0,2\}$};
    \node(12) at (4.25,2.25) {$\{1,2\}$};
    \node(012) at (3.25,3.25) {$\{0,1,2\}$};
    %
    \draw[->, thick] (0) -- (E);
    \draw[->, thick] (E) -- (1);
    \draw[->, thick] (1) -- (01);
    \draw[->, thick] (0) -- (01);
    %
    \draw[-, thick] (2) -- (02);
    \draw[-, thick] (2) -- (12);
    \draw[-, thick] (12) -- (012);
    \draw[-, thick] (02) -- (012);    
    %
    \draw[->, dashed, thick] (012) -- (01);
    \draw[->, dashed, thick] (0) -- (02);
    \draw[->, dashed, thick] (2) -- (E);
    \draw[->, dashed, thick] (1) -- (12);
\end{tikzpicture}
\caption{We illustrate $\mathcal{G}_2$, with the ``green'' diamond on the lower left and ``yellow'' diamond on the upper right. The controlling edges (dashed) connecting the two colored diamonds.} 
\label{fig:RiRi}
\end{figure}

The final result of this section establishes the 
existence of a suitable function $\sigma$ to use in our exact sequences.
First some additional terminology:
an orientation on $\mathcal{G}_{\vav}$ is \emph{positively swapped} if the subgraph induced on the vertices labeled 
$\{a,b\}$, $\{a\}$, and $\{b\}$ is a path of length 2 if, and only if, $0<a<b\leq \vav$. 

\begin{lem}
\label{lem:good-graph}
For every $\vav\geq 1$, there exists an oddly acyclic and positively 
swapped orientation for $\mathcal{G}_{\vav}$ such that for every $A\in\binom{[\vav]}{2}$, there is a directed edge from $A$ to $\max(A)$. 
\end{lem}

\begin{proof}
We prove this by induction. 
From \nameref{lem:RiRi}, there are 8 oddly acyclic orientations for $\mathcal{G}_1$, exactly 4 are of those positively swapped. 
Of the 4 orientations, exactly 2 have a directed edge from $\{0,1\}$ to $\{1\}$. 

Now suppose the subgraph $\mathcal{G}_{\vav-1}$ is oddly acyclic and positively swapped such that for every $A\in\binom{[\vav-1]}{2}$, there is a directed edge from $A$ to $\max(A)$.  
The only choice we impose to get positively swapped is that the controlling edge between $\{\vav\}$ and $\emptyset$ points in a direction such that the subgraph on $\{0\}$, $\{\vav\}$, and $\emptyset$ is a path of length 2. 
By induction, the subgraph on $\{0\}$, $\{a\}$, and $\emptyset$ is a path of length 2, for all $0<a<\vav$. 
Therefore, regardless of the choice of orientation for the remaining controlling edges, the resulting orientation from \nameref{lem:RiRi} will be positively swapped. 
To ensure the last property, we choose the orientation for the controlling edges so that the edge points from $\{a,\vav\}$ to $\{\vav\}$, for all $a\in[\vav-1]$. 
\end{proof}

\section{Proof of Theorems~\ref{thm:exact-sequencesA} 
and~\ref{thm:exact-sequencesB}}
\label{sec:proofs}
Both theorems follow easily from our work in Sections~\ref{sec:exact}
and~\ref{sec:Rihanna}.

\begin{lem}
\label{lem:good-sigma}
For $\vav\geq 1$, there exists a function $\sigma$ such that
\begin{enumerate}
    \item[$(i)$] for all $A\subset[\vav]$ and $b\notin A$,
    \[ \sigma(A, A\cup b)=\pm 1, \]
    \item[$(ii)$] for all $0\leq a < b\leq \vav$, 
    \begin{align*} 
    \sigma(a, a\cup b) = \left\{ \begin{array}{ll} 
    \sigma(b, a\cup b) & 0=a,\\
    -\sigma(b, a\cup b) & 0<a,
    \end{array}\right. &&
    \sigma(a, a\cup b) = \left\{ \begin{array}{ll} 
    1 & 0=a,\\
    -1 & 0<a,
    \end{array}\right.
    \end{align*}
    \item[$(iii)$] for all $C\in\binom{[\vav]}{k-1}$ and distinct $a,b\notin C$,
    \[ \sigma(C, C\cup a)\sigma(C\cup a, C\cup\{a,b\}) + \sigma(C, C\cup b)\sigma(C \cup b, C\cup\{a,b\})=0. \]
\end{enumerate}
\end{lem}

\begin{proof}
Translating the directions on edges in an oriented graph $\mathcal{G}_{\vav}$
to values $\pm 1$---as described at the start of Section~\ref{sec:Rihanna}---each such graph 
encodes some function $\sigma$ with the correct domain and range in property (i).
By Lemma~\ref{lem:good-graph}, the orientation on $\mathcal{G}_{\vav}$ can be chosen so that properties (ii) and (iii) hold. 
\end{proof}

\begin{proof}[Proof of Theorems~\ref{thm:exact-sequencesA} and~\ref{thm:exact-sequencesB}]
By Lemmas~\ref{lem:well-defined},~\ref{lem:sigma-hom},~\ref{lem:exact},~\ref{lem:good-sigma}, 
and full nondegeneracy of $S$, the homomorphisms 
$\Upsilon^k$ form an exact sequence. 
Using Fact~\ref{fact:aut-der}, if $P\in\{D, G\}$, then the zero set $\Op{S}{P}$ is either $\Der(S)$ or $\Aut(S)$. 
In the case $P=G$, apply Lemma~\ref{lem:unit-embedding} to obtain the sequence in Theorem~\ref{thm:exact-sequencesB} from the zero-sets. 
\end{proof}

\section{Applications \& Examples}
\label{sec:app-ex}
The techniques developed in Sections~\ref{sec:seq} and~\ref{sec:exact} to prove 
Theorems~\ref{thm:exact-sequencesA} and~\ref{thm:exact-sequencesB} being rather new,
it might not be clear to the reader how they can be used to study tensors, nor 
how they might be applied in more general settings. We therefore devote this
final section to a range of examples and applications that illustrate the scope and 
potential of these new tools.


\subsection{Entangled quantum states}
\label{subsec:quantum}
In~\cite{DVC}, D\"ur-Vidal-Cirac considered
two maximally entangled quantum states on 3 qubits,
and demonstrated they are inequivalent.
Here, we provide an elementary verification of
their discovery using our exact sequences.
Let $\mathbb{H}$ be an 8-dimensional Hilbert space with 
$\langle \cdot | : \mathbb{H} \rightarrow \C^2 \oslash \C^2 \oslash \C^2$
(more commonly represented as $\mathbb{C}\oslash(\mathbb{C}^2\otimes \mathbb{C}^2\otimes\mathbb{C}^2)$, 
the dual space to $\mathbb{C}^{2}\otimes \mathbb{C}^2\otimes \mathbb{C}^2$).
In the usual convention, $\mathbb{C}^2=\mathbb{C}\bra{0}\oplus \mathbb{C}\bra{1}$,
so a basis for $\mathbb{H}$ is $\bra{abc}$ with $a,b,c\in\{0,1\}$.
The {\em Greenberger-Horne-Zeilinger} state is 
defined
\begin{align*}
	\bra{GHZ} = \frac{\sqrt{2}}{2}(\langle 000 | + \langle 111 |),
\end{align*}
while the state $W$ from \cite{DVC} is
\begin{align*}
	\bra{W} = \frac{\sqrt{3}}{3}(\bra{100} + \bra{010} + \bra{001}).
\end{align*}
An elementary calculation reveals that 
$\Cen(GHZ) \cong \C^2$, and the sequence is
\[ 0 \longrightarrow \C^2 \longrightarrow \C^2\oplus \C^2 \oplus \C^2 \longrightarrow \C^2\oplus\C^2 \longrightarrow 0. \]
On the other hand, $\Cen(W) \cong \mathbb{C}[x]/(x^2)$.
Even though the centroid is also $2$-dimensional,
the algebra is fundamentally different---in particular, 
it has a nontrivial Jacobson radical. 
Furthermore, the corresponding sequence for 
$W$ contains nontrivial outer derivations; 
with $A=\C[x] / (x^2)$, we have
\[ 0 \longrightarrow A \longrightarrow A\oplus A \oplus A \longrightarrow \C\oplus A \oplus A \longrightarrow \C \longrightarrow 0. \]
The two states are clearly inequivalent.

\subsection{Composition and matrix products}\label{sec:mat}
This example uses the familiar concept of \emph{tensor contraction}
(also known as \emph{hyper-matrix multiplication}), which for convenience we model
as a special case of composition in a module category. 

Recall in our notation $V\oslash U=\hom(U,V)$, which is again a
module.  Hence, composition of functions in the $K$-module category 
is a bilinear map
\begin{align*}
	\circ & :A\oslash B \times B\oslash C\bmto A\oslash C.
\end{align*}
If $A=K^a$, $B=K^b$, and $C=K^c$, we can, after fixing
bases, identify composition with the {\em matrix multiplication
tensor}
\begin{align*}
	\mathbb{M}_{a\times b}(K)\times \mathbb{M}_{b\times c}(K)\bmto 
		\mathbb{M}_{a\times c}(K).
\end{align*}
The latter has been studied extensively, and
its derivations and automorphisms are known. Thus, it is a good example
to illustrate our methods, as they provide another means to see
the structure.  

Within composition we have three self-evident 
contributions to the nuclei:
\begin{align*}
	\Nuc_{20}(\circ) & \cong \End(A), &
	\Nuc_{21}(\circ) & \cong \End(B), & 
	\Nuc_{10}(\circ) & \cong \End(C). 
\end{align*}
Therefore, while
$B$ does not occur in the codomain, its influence in the middle of the
domain can be identified by the nuclei.  This may not seem surprising
given how we introduced the product, but such a tensor
could be given as a black-box.  Then the product would take the form
$K^{ab}\times K^{bc}\bmto K^{ac}$, and that is a completely ambiguous
decomposition. Let us consider a specific example.

\begin{ex}
If $t \colon \M_{2\times 3}(\C) \times \M_{3\times 4}(\C) \bmto \M_{2\times 4}(\C)$, 
then 
\begin{align*}
   \Cen_3(t) &\cong \C, & \Nuc(t) &\cong \M_2(\C) \oplus \M_3(\C) \oplus \M_4(\C).
\end{align*} 
Therefore, the sequences in Theorem~\ref{thm:exact-sequencesA} and Theorem~\ref{thm:exact-sequencesB} have the form
\begin{align*} 
0 \longrightarrow \C \longrightarrow \M_2 \oplus \M_3 \oplus \M_4 \longrightarrow (\gl_2 \oplus \gl_3 \oplus \gl_4)/\C \longrightarrow 0, \\
1 \longrightarrow \C^\times \longrightarrow \GL_2 \times \GL_3 \times \GL_4 \longrightarrow \SL_3 \rtimes (\GL_2 \times \GL_4) \longrightarrow 1.
\end{align*}
In other words, all autotopisms of tensors given by matrix multiplication are ``inner,'' in the sense that they are realized as the 
groups of units of the various nuclei. 
\end{ex}

\subsection{Non-associative tensor decompositions}
Tensor contraction is expressible in other ways,
partly because it relates to our familiar associative matrix multiplication.
However, tensors can also be constructed in non-associative ways;
one need only consider products of Lie algebras to witness such cases.
Our next examples demonstrate how we detect nonassociative components of a
tensor.

\begin{ex}\label{ex:cubic-sl2}
Let $K$ be a field such that $2K=K$.
Let $\langle t| \colon K^4 \times K^4 \times K^4 \rightarrowtail \bigwedge^3 K^4$ be given by the exterior cube of $K^4$. 
Also, let $\langle s| : K\times \sl_2 \times \sl_2 \rightarrowtail \sl_2$ be the tensor that maps $(k, X, Y)$ to $k[X,Y]$. 
Define the tensor product (over $K$) of $K$-tensors $t$ and $s$ as 
$\langle t\otimes s |\colon (K^4\otimes K) \times 
(K^4\otimes \sl_2) \times (K^4\otimes \sl_2) \rightarrowtail K^4\otimes \sl_2$, where
\[ \langle t\otimes s | u\otimes v \rangle = \langle t | u \rangle \otimes \langle s | v\rangle. \]

A calculation shows that $\Der(t) \cong K^2\oplus \gl_4(K)$. 
For all other $A\subset [3]$, with $|A|\geq 2$, $\Der_A(t)\cong K^{|A|-1}$.
For $s$, on the other hand, $\Der_{\{0,1,2\}}(s) \cong K\oplus \gl_3(K)$ and $\Der(s) \cong K\oplus \Der_{\{0,1,2\}}(s)$. 
For all other $A\subset [3]$, $\Der_A(s)\cong K^{|A|-1}$. 
These derivation algebras are detected by the sequence in Theorem~\ref{thm:exact-sequencesA} for $t\otimes s$. In fact, there are no larger algebras:
\begin{align*}
0 \longrightarrow K \longrightarrow K^4 \longrightarrow K^6 \longrightarrow K\oplus \gl_3\oplus \gl_4 \longrightarrow \sl_3\oplus \sl_4 \longrightarrow 0.
\end{align*}
\end{ex}

\begin{ex}
Let $K$ be a field with degree 2 and 3 extensions denoted by $E$ and $F$ respectively. 
Let $t$ and $s$ be the dot products on $E^2$ and $F^2$, respectively. 
We concatenate a 1-dimensional coordinate (over $K$) to both $t$ and $s$, making them
$K$-trilinear: $\langle t' | :K\times K^4\times K^4 \rightarrowtail K^2$ and 
$\langle s' | :K^6\times K^6\times K\rightarrowtail K^3$.  
Let $r=t'\otimes s'$. 
Instead of writing out the sequence like in the previous examples, we display the dimension of 
every algebra over $K$ in Figure~\ref{fig:dim-seq}.
While $r$ is only $K$-trilinear (the centroid of $r$ is isomorphic to $K$), the local centroids 
detect the fact that $r$ was built from tensors that are \emph{bi}-linear over extensions, namely degree 2 and 3 extensions. 

\begin{figure}[h]
\centering
\begin{tikzpicture}
    \node(start) at (0,0) {0};
    \node(start0) at (-3.5, 0) {{\footnotesize 0}};
    \node(0123) at (0,0.75) {1};
    \node(labelC4) at (-3.5, 0.75) {{\footnotesize $\Cen_4(r)$}};
    \node(012) at (-1.5, 1.5) {2};
    \node(013) at (-0.5, 1.5) {1};
    \node(023) at (0.5, 1.5) {3};
    \node(123) at (1.5, 1.5) {1};
    \node(labelC3) at (-3.5, 1.5) {{\footnotesize $\Cen_3(r)$}};
    \node(01) at (-2, 2.25) {2};
    \node(02) at (-1.2, 2.25) {6};
    \node(03) at (-0.4, 2.25) {3};
    \node(12) at (0.4, 2.25) {8};
    \node(13) at (1.2, 2.25) {1};
    \node(23) at (2, 2.25) {12};
    \node(labelN) at (-3.5, 2.25) {{\footnotesize $\Nuc(r)$}};
    \node(D) at (0, 3) {26};
    \node(labelD) at (-3.5, 3) {{\footnotesize $\Der(r)$}};
    %
    \draw[-] (start) -- (0123);
    \draw[-] (0123) -- (012);
    \draw[-] (0123) -- (013);    
    \draw[-] (0123) -- (023);
    \draw[-] (0123) -- (123); 
    \draw[-] (012) -- (01); 
    \draw[-] (012) -- (02); 
    \draw[-] (012) -- (12); 
    \draw[-] (013) -- (01);
    \draw[-] (013) -- (03); 
    \draw[-] (013) -- (13);
    \draw[-] (023) -- (02);
    \draw[-] (023) -- (03);
    \draw[-] (023) -- (23);
    \draw[-] (123) -- (12);
    \draw[-] (123) -- (13);
    \draw[-] (123) -- (23);
    \draw[-] (01) -- (D);
    \draw[-] (02) -- (D);
    \draw[-] (03) -- (D);
    \draw[-] (12) -- (D);
    \draw[-] (13) -- (D);
    \draw[-] (23) -- (D);
    \draw[->] (start0) -- (labelC4);
    \draw[->] (labelC4) -- (labelC3);
    \draw[->] (labelC3) -- (labelN);
    \draw[->] (labelN) -- (labelD);
\end{tikzpicture}
\caption{A graphical description of the sequence in Theorem~\ref{thm:exact-sequencesA}. 
Here, we have separated the direct summands of the terms of the sequence, and we are only displaying their dimensions over $K$. 
The sequence starts at the bottom and goes to the top, with the last nontrivial term being the 26-dimensional derivation algebra.
The vertical sequence on the left aligns with the dimensions of the direct summands---in lex-least order.}
\label{fig:dim-seq}
\end{figure}
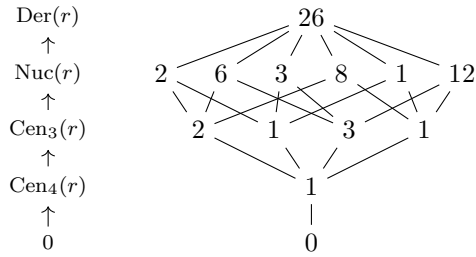
\end{ex}

\section*{Acknowledgments}
We thank the anonymous referee for 
helping us to clarify the relationship between our
sequences for tensors and those already in the 
literature for algebras.


\end{document}